\documentclass[12pt]{article}

\usepackage{mathrsfs}

\usepackage{amsmath,amsthm,amssymb}
\usepackage{graphicx}
\usepackage{vector}

\font\tencmmib=cmmib10 \skewchar\tencmmib '60
\newfam\cmmibfam
\textfont\cmmibfam=\tencmmib

\font\tenmsb=msbm10

\def\Bbb#1{\hbox{\tenmsb#1}}

\def\bbox{\quad\hbox{\vrule \vbox{\hrule \vskip2pt \hbox{\hskip2pt
\vbox{\hsize=1pt}\hskip2pt} \vskip2pt\hrule}\vrule}}
\def\lessim{\ \lower4pt\hbox{$
\buildrel{\displaystyle <}\over\sim$}\ }
\def\gessim{\ \lower4pt\hbox{$\buildrel{\displaystyle >}
\over\sim$}\ }

\def\eps{{\varepsilon}}

\def\qed{\hfill\break\rightline{$\bbox$}}

\newcommand{\veps}{\vec{\varepsilon}}
\newcommand{\vsi}{{\vec{\sigma}}}
\newcommand{\vrho}{{\vec{\rho}}}

\newcommand{\vtau}{\vec{\tau}}
\newcommand{\veta}{\vec{\eta}}
\newcommand{\vx}{\vec{x}}

\newcommand{\vz}{\vec{z}}

\newtheorem{lemma}{\bf Lemma}

\newtheorem{theorem}{\bf Theorem}

\newtheorem{proposition}{\bf Proposition}

\newenvironment{Proof of lemma}{\noindent{\bf Proof of Lemma}}{\hfill$\Box$\newline}
\newenvironment{Proof of theorem}{\noindent{\bf Proof of Theorem}}{\hfill{\footnotesize${\square}$}\newline}
\newenvironment{Proof of theorems}{\noindent{\bf Proof of Theorems}}{\hfill$\Box$\newline}
\newenvironment{Proof of proposition}{\noindent{\bf Proof of Proposition}}{\hfill$\Box$\newline}
\newenvironment{Proof of propositions}{\noindent{\bf Proof of Propositions}}{\hfill$\Box$\newline}
\newenvironment{Proof of exercise}{\noindent{\it Proof of Exercise:}}{\hfill$\Box$}

\numberwithin{equation}{section}

\font\tencmmib=cmmib10 \skewchar\tencmmib '60
\newfam\cmmibfam
\textfont\cmmibfam=\tencmmib

\font\tenmsb=msbm10

\def\Bbb#1{\hbox{\tenmsb#1}}

\def\bbox{\quad\hbox{\vrule \vbox{\hrule \vskip2pt \hbox{\hskip2pt
\vbox{\hsize=1pt}\hskip2pt} \vskip2pt\hrule}\vrule}}
\def\lessim{\ \lower4pt\hbox{$
\buildrel{\displaystyle <}\over\sim$}\ }
\def\gessim{\ \lower4pt\hbox{$\buildrel{\displaystyle >}
\over\sim$}\ }

\def\eps{\varepsilon}

\def\go0{\to 0}

\def\leftitem#1{\item{\hbox to\parindent{\enspace#1\hfill}}}

\def\qed{\hfill\break\rightline{$\bbox$}}

\def\sg{\sigma}

\def\sg2{\sigma^2}

\def\__{_{\infty}}

\begin{document}

\title{The Aizenman-Sims-Starr scheme and Parisi formula for mixed $p$-spin spherical models}
\author{Wei-Kuo Chen\footnote{Department of Mathematics, University of California at Irvine, email: weikuoc@uci.edu.}}

\maketitle

\begin{abstract}  
The Parisi formula for the free energy in the spherical models with mixed even $p$-spin interactions was proven in Michel Talagrand \cite{Tal061}. 
In this paper we study the general mixed $p$-spin spherical models including $p$-spin interactions for odd $p$. We establish the Aizenman-Sims-Starr scheme and from this together with many well-known results and Dmitry Panchenko's recent proof on the Parisi ultrametricity conjecture \cite{Pan11:4}, we prove the Parisi formula.
\end{abstract}

{Keywords: Parisi formula, ultrametricity}

\section{Introduction and main results.}

In the past decades, the Sherrington-Kirkpatrick model \cite{SK72} of Ising spin glass has been intensively studied  with the aim of understanding the strange magnetic behavior of certain alloys. Despite only being a mean-field model, the formula for its free energy has an intricate and beautiful variational formula, first discovered by G. Parisi \cite{Parisi79,Parisi80}, based on a hierarchical replica symmetry breaking scheme. As it is hard to compute the free energy from the Parisi formula, the closely related spherical model was studied by A. Crisanti and H. J. Sommers \cite{CS92}; this model is widely believed to retain the main features of the Ising case, and the analogous Parisi formula admits a more explicit representation for the free energy.

Mathematically, the Parisi formula for the Ising and spherical spin glass models was proven rigorously and generalized to mixtures of $p$-spin interactions  in the seminal works of M. Talagrand \cite{Tal061,Tal06} following the discovery of the replica symmetry breaking interpolation scheme by F. Guerra \cite{Guerra03}. For technical reasons, only the mixture of even $p$-spin interactions was considered. The main missing ingredient in proving the general case was the validity of the ultrametric property for the Gibbs measure, which was recently verified by D. Panchenko \cite{Pan11:4}. This together with a series of known results led to the Parisi formula for the Ising spin glass model with general mixed $p$-spin interactions \cite{Pan11:1}.

In view of the approach in \cite{Pan11:1}, the most crucial step is the derivation of the inequality that the lower limit of the free energy is bounded from below by the Parisi formula. It is well-known, in the case of mixed even $p$-spin models, that the limiting free energy can be represented by a new variational formula through the Aizenman-Sims-Starr (A.S.S.) scheme \cite{ASS03}. An important implication of this scheme is that it provides a natural lower bound for the lower limit of the free energy for {\it general} mixed $p$-spin models. Since under a small perturbation of the Hamiltonian the asymptotic Gibbs measures in the A.S.S. scheme satisfy the Ghirlanda-Guerra (G.G.) identities, the main result in \cite{Pan11:4} implies that the support of these measures are asymptotically ultrametric. This property then allows us to characterize the asymptotic Gibbs measure via the Poisson-Dirichlet cascades and consequently recover the Parisi formula in the general mixed $p$-spin case.

In the present paper, we will establish the Aizenman-Sims-Starr (A.S.S.) scheme and prove the Parisi formula in the spherical model with general mixed $p$-spin interactions taking the approach of \cite{Pan11:1}.  However, the task turned out to be much more difficult, mostly, due to the fact that uniform measure on the sphere is not a product measure. In particular, as we shall see in the proof of Theorem $\ref{ASS:prop1}$ below, the establishment of the A.S.S. scheme in this case becomes much more intricate since in general its construction requires a certain decoupling of the configuration space. Also, the approximation of the Parisi formula using the ultrametricity of the Gibbs measure will be more involved. These obstacles are the main concerns of the paper and will be studied in the next two sections.

Let us now state the mixed $p$-spin spherical model. For given $N\geq 1,$ we denote by $S_N=\{\vsi\in\mathbb{R}^N:\|\vsi\|=\sqrt{N}\}$ the  configuration space and by $\lambda_N$ its normalized surface measure. Consider the pure $p$-spin Hamiltonians $H_{N,p}(\vsi)$ for $p\geq 1$ indexed by $\vsi\in S_N,$
\begin{align}\label{eq1}
H_{N,p}(\vsi)=\frac{1}{N^{(p-1)/2}}\sum_{1\leq i_1,\ldots,i_p\leq N}g_{i_1,\ldots,i_p}\sigma_{i_1}\cdots\sigma_{i_p},
\end{align}
where the random variables $g_{i_1,\ldots,i_p}$ for all $p\geq 1$ and all $(i_1,\ldots,i_p)$ are i.i.d. standard Gaussian. The Hamiltonian of the mixed $p$-spin spherical model is given by
\begin{align}
\label{eq2}
H_N(\vsi)=\sum_{p\geq 1}\beta_p H_{N,p}(\vsi)
\end{align}
with coefficients $(\beta_p)$ that satisfy $\beta_p\geq 0$ and decrease fast enough, $\sum_{p\geq 1}2^p\beta_p^2<\infty.$ This technical assumption ensures that on $S_N$ the series $(\ref{eq2})$ converges a.s. and the covariance of the Gaussian process
is given by a function of the normalized scalar product, called the overlap, $R_{1,2}=N^{-1}\sum_{i\leq N}\sigma_i^1\sigma_i^2$
of spin configurations $\vsi^1$ and $\vsi^2,$
$$
\mathbb{E}H_N(\vsi^1)H_N(\vsi^2)=N\xi(R_{1,2}),
$$
where $\mathbb{E}$ is the expectation with respect to all Gaussian randomness and $\xi(x)=\sum_{p\geq 1}\beta_p^2 x^p.$ 

Now we formulate the Parisi formula. Given $k\geq 1$, let $\mathbf{m}=(m_\ell)_{0\leq \ell\leq k}$ and $\mathbf{q}=(q_\ell)_{0\leq \ell\leq k+1}$ be two sequences satisfying
\begin{align*}
&0=m_0\leq m_1\leq\ldots\leq m_{k-1}\leq m_k=1,\\
&0=q_0\leq q_1\leq \ldots \leq q_k\leq q_{k+1}=1.
\end{align*}
One may think of the triplet $(k,\mathbf{m},\mathbf{q})$ as a distribution function $\vx$ on $\left[0,1\right]$ with $\vx(q)=m_\ell$ if $q_\ell\leq q<q_{\ell+1}$ for $0\leq \ell\leq k$ and $\vx(1)=1.$ We define for $1\leq \ell\leq k,$
$$d_\ell=\sum_{\ell\leq p\leq k}m_p(\xi'(q_{p+1})-\xi'(q_p)).$$ 
For a given parameter $b>d_1,$ we define $D_\ell=b-d_\ell$ for $1\leq \ell\leq k$ and $D_{k+1}=b.$ Let $\theta(x)=x\xi'(x)-\xi(x).$ 
The formulation of the Parisi formula involves a functional $\mathcal{P}$ of the distribution function $\vx$,
\begin{align*}
\mathcal{P}(\vx)=\inf_{b>d_1}\frac{1}{2}\left(b-1-\log b+\frac{\xi'(q_1)}{D_1}+\sum_{1\leq \ell\leq k}\frac{1}{m_\ell}\log \frac{D_{\ell+1}}{D_\ell}-\sum_{1\leq \ell\leq k}m_\ell(\theta(q_{\ell+1})-\theta(q_{\ell}))\right),
\end{align*}
which is the limit of the Parisi functional described below. For every $M\geq 1,$ let $\vz_p=(z_{p,1},\ldots,z_{p,M})$ for $0\leq p\leq k$ be centered Gaussian random vectors with $\mathbb{E}z_{p,i}z_{p,i'}=\delta_{i,i'}(\xi'(q_{p+1})-\xi'(q_p))$ for any $i\neq i'$ and let $(\vz_p)_{0\leq p\leq k}$ be independent of each other. Starting from
\begin{align*}
X_{k+1}^M=\log\int_{S_M}\exp \veps\cdot \left(\vz _0+\cdots+\vz _k\right)d\lambda_M(\veps),
\end{align*}
we define decreasingly for $0\leq p\leq k,$
\begin{align*}
X_p^M=\frac{1}{m_p}\log\mathbb{E}_p \exp m_p X_{p+1}^M,
\end{align*}
where $\mathbb{E}_p$ denotes the expectation in the random vectors $(\vz_i)_{i\geq p}.$ When $m_p=0,$ this means $X_p^M=\mathbb{E}_p X_{p+1}^M.$ 
Set the Parisi functional $$
\mathcal{P}_M(x)=X_0^M-\frac{1}{2}\sum_{1\leq p\leq k}m_p(\theta(q_{p+1})-\theta(q_p)).
$$
A series of nontrivial applications of elementary large deviation principle implies $\mathcal{P}(\vx)=\lim_{M\rightarrow\infty}\mathcal{P}_M(\vx)$ (see Proposition 3.1 \cite{Tal061}). Let us denote the partition function of the model by $$
Z_N=\int_{S_N}\exp H_N(\vsi)d\lambda_N(\vsi).
$$
\begin{theorem}[The Parisi formula]\label{thm1} We have
\begin{align}\label{thm1:eq1}
\lim_{N\rightarrow\infty}\frac{1}{N}\mathbb{E}\log Z_N=\inf_{\vx}\mathcal{P}(\vx),
\end{align}
where the infimum is taken over all possible choices of $\vx$.
\end{theorem}

\noindent 
The quantity in the limit on the left-hand side of $(\ref{thm1:eq1})$ is usually called the free energy of this model and the formula on the right-hand side of $(\ref{thm1:eq1})$ is the famous Parisi formula in the spherical form, which is also known to have the Crisanti-Sommers representation \cite{CS92} (see Theorem 4.1 in \cite{Tal061}). Of course, one may also consider the model in the present of external field and prove the equation $(\ref{thm1:eq1}),$  but for the clarity of our discussion, we will only consider the Hamiltonian of the form $(\ref{eq2}).$

\smallskip

\noindent{\bf Acknowledgement.} The author would like to thank Dmitry Panchenko for motivating this work and his advisor, Michael Cranston, for many helpful discussions. Special thanks are due to Alexander Vandenberg-Rodes for several helpful suggestions regarding the presentation of the paper.


\section{The A.S.S. scheme for spherical models.}

\noindent The main result of this section is the construction of the A.S.S. scheme for spherical models. For $M,N\geq 1$, let us denote by $(\vz(\vsi)=(z_1(\vsi),\ldots,z_M(\vsi)):\vsi\in S_N)$ and $(y(\vsi):\vsi\in S_N)$ two independent families of centered Gaussian random variables with covariances
\begin{align}
\begin{split}\label{ASS:eq0.1}
\mathbb{E}z_i(\vsi^1)z_{i'}(\vsi^2)&=\delta_{i,i'}\xi'(R_{1,2})
\end{split}\\
\begin{split}\label{ASS:eq0.2}
\mathbb{E}y(\vsi^1)y(\vsi^2)&=\theta(R_{1,2})
\end{split}
\end{align}
for every $\vsi^1,\vsi^2\in S_N.$ These random variables are independent of all the other randomness. Let $G_{M,N}$ be the Gibbs measure corresponding to the Hamiltonian 
\begin{align}
\label{ASS:eq1}
H_{M,N}(\vsi)&=\sum_{p\geq 1}\frac{\beta_p}{(M+N)^{(p-1)/2}}\sum_{1\leq i_1,\ldots,i_p\leq N}g_{i_1,\ldots,i_{p}}\sigma_{i_1}\cdots\sigma_{i_p}
\end{align}
with reference measure $\lambda_N$ on $S_N$ and let $\left<\cdot\right>_{M,N}$ be the Gibbs average of $G_{M,N}.$
The following is our main result.

\begin{theorem}[The A.S.S. scheme]
\label{ASS:prop1}
For any $\delta>0$ and $M\geq 1,$ the lower limit of $N^{-1}\mathbb{E}\log Z_N$ is bounded from below by
\begin{align}
\begin{split}\label{ASS:prop1:eq1}
&\frac{1}{M}\liminf_{N\rightarrow\infty}\left(\mathbb{E}\log \int_{S_M}\left<\exp \veps\cdot \vz (\vsi)\right>_{M,N}d\lambda_M(\veps)-
\mathbb{E}\log \left<\exp \sqrt{M}y(\vsi)\right>_{M,N}\right)\\
&\qquad\qquad -\delta\xi'(1)+\frac{1}{M}\log \nu_M(A_\delta),
\end{split}
\end{align}
where $\nu_M$ is the probability measure of the $M$-dimensional Gaussian density and 
$$
A_\delta=\{\veps\in\mathbb{R}^M:M\leq\|\veps\|^2\leq M(1+\delta)\}.
$$
\end{theorem}

\smallskip

\noindent 
Note that a standard large deviation result yields $\lim_{M\rightarrow \infty}M^{-1}\log\nu_M(A_\delta)=0$ for every $\delta>0.$ Thus $(\ref{ASS:prop1:eq1})$ gives a lower bound for the lower limit of the free energy that will play an essential role in proving the Parisi formula in Section $\ref{proof}.$ Suppose that $\mu_N$ is the uniform probability measure on $\Sigma_N=\left\{-1,+1\right\}^N$. If we replace the measure space $(S_N,\lambda_N)$ by $(\Sigma_N,\mu_N)$ and ignore the error term $-\delta\xi'(1)+M^{-1}\log \nu_M(A_\delta)$ in $(\ref{ASS:prop1:eq1})$, then Theorem \ref{ASS:prop1} gives the A.S.S. scheme for the Ising spin glass model with general mixed $p$-spin interactions (see e.g. Section 15.8 \cite{Tal11}). 

We will take the spirit in \cite{ASS03} to establish the A.S.S. scheme. However, since uniform measure on the spherical configuration space in the spherical models is not a product measure, we have to seek a proper way to approximate this measure by another one that possesses a certain product structure. To deal with this obstacle, we will use the idea of Poincar\'{e}'s lemma (see page 77 \cite{ST93}) that the projection on any fixed $M$ coordinates of the uniform probability measure $\lambda_{M+N}$ on the sphere $S_{M+N}$ converge weakly to $M$-dimensional standard Gaussian measure as $N$ tends to infinity. The argument of our proof proceeds as follows. Recall the definition of $H_{M+N}$ from $(\ref{eq2})$. For $\vrho\in S_{M+N},$ let us write $\vrho=(\vsi,\veps)$ for the first $N$ coordinates $\vsi=(\sigma_1,\ldots,\sigma_N)$ and the last $M$ cavity coordinates $\veps=(\varepsilon_1,\ldots,\varepsilon_M)$ and write
\begin{align*}
H_{M+N}(\vrho)&=H_{M,N}(\vsi)+\sum_{i\leq M}\varepsilon_i Z_i(\vsi)+\gamma(\vrho),
\end{align*}
where $H_{M,N}$ is given in $(\ref{ASS:eq1}),$ the term $\varepsilon_iZ_i(\vsi)$ consists of all terms in $H_{M+N}$ with only one factor $\varepsilon_i$ from $\veps$ present, and the last term is the sum of terms with at least two factors in $\veps.$ From now on, without causing any ambiguity, for any given vectors $\vec{\omega},\vec{\omega}'\in\Bbb{R}^K$ for some $K\geq 1$, we define $R(\vec{\omega},\vec{\omega}')=K^{-1}\sum_{i\leq K}\omega_i\omega_i'.$ One can check easily that
\begin{align}
\begin{split}
\label{ASS:eq2.1}
\mathbb{E}H_{M,N}(\vsi^1)H_{M,N}(\vsi^2)&=(M+N)\xi\left(\frac{N}{M+N}R(\vsi^1,\vsi^2)\right),
\end{split}\\
\begin{split}
\label{ASS:eq2.2}
\mathbb{E}Z_i(\vsi^1)Z_i(\vsi^2)&=\xi'\left(\frac{N}{M+N}R(\vsi^1,\vsi^2)\right),
\end{split}\\
\begin{split}
\label{ASS:eq2.3}
\mathbb{E}\gamma(\vrho)^2&\leq \frac{M^2}{M+N}\xi''(1).
\end{split}
\end{align}
Let us notice an elementary fact that for a real-valued sequence $(c_N)$, if $\liminf_{N\rightarrow\infty}c_N/N=r\in\mathbb{R},$ then $M^{-1}\liminf_{N\rightarrow\infty}(c_{M+N}-c_N)\leq r$ for every $M\geq 1.$ This implies 
\begin{align}\label{ASS:main:eq0}
&\liminf_{N\rightarrow\infty}\frac{\mathbb{E}\log Z_N}{N}\geq\frac{1}{M}\liminf_{N\rightarrow\infty}\mathbb{E}\log \frac{Z_{M+N}}{Z_N}
\end{align}
and thus to obtain Theorem $\ref{ASS:prop1}$, it suffices to prove that $(\ref{ASS:prop1:eq1})$ is the lower bound of the right-hand side of $(\ref{ASS:main:eq0})$.
Let us write $\mathbb{E}\log{Z_{M+N}}/{Z_N}$ as
\begin{align}
\begin{split}\label{ASS:main:eq1}
&\mathbb{E}\log \int_{A_\delta}\left< \exp\veps\cdot \vz(\vsi)\right>_{M,N}d\nu_M(\veps)-\mathbb{E}\log\frac{\int_{S_N}\exp H_Nd\lambda_N(\vsi)}{\int_{S_N}\exp H_{M,N}d\lambda_N(\vsi)}\\
&+\mathbb{E}\log\frac{Z_{M+N}}{\int_{A_\delta}J_{M,N}d\nu_M(\veps)},
\end{split}
\end{align}
where  
\begin{align*}
J_{M,N}(\veps)&=\int_{S_N}\exp(H_{M,N}(\vsi)+\veps\cdot \vz (\vsi))d\lambda_N(\vsi),\quad\veps\in\mathbb{R}^M.
\end{align*}
In what follows, we prove several lemmas that will be used to bound these three terms in $(\ref{ASS:main:eq1})$ as $N$ tends to infinity.

\begin{lemma}
\label{ASS:lem5} For $\vz\in\mathbb{R}^M,$ we have
\begin{align}
\label{ASS:lem5:eq1}
\int_{A_\delta}\exp \veps\cdot\vz d\nu_M(\veps)\geq  
\nu_M\left(A_\delta\right)\int_{S_M}\exp\veps\cdot \vz d\lambda_M(\veps).
\end{align}
\end{lemma}

\begin{proof}
Using the rotational invariance of $\nu_M,$ there  exists a probability measure $\gamma_M$ on $\mathbb{R}^+$ such that
$\nu_M$ is the image of $\gamma_M\times \lambda_M$ under the map $(s,\veps)
\mapsto s\veps.$ For $1\leq s\leq (1+\delta)^{1/2}$, using symmetry,
\begin{align*}
\int \exp s\veps\cdot\vz d\lambda_M(\veps)&=\int \cosh(s\veps\cdot\vz)d\lambda_M(\veps)\\
&\geq \int \cosh(\veps\cdot\vz)d\lambda_M(\veps)=\int \exp \veps\cdot\vz d\lambda_M(\veps).
\end{align*} 
Therefore, $(\ref{ASS:lem5:eq1})$ holds by
\begin{align*}
\int_{A_\delta}\exp \veps\cdot\vz d\nu_M(\veps)&=\int_{1}^{(1+\delta)^{1/2}}\int_{S_M}\exp s\veps\cdot \vz d\lambda_M(\veps)d\gamma_M(s)\\
&\geq\int_{1}^{(1+\delta)^{1/2}}\int_{S_M}\exp\veps\cdot \vz d\lambda_M(\veps)d\gamma_M(s)\\
&= \gamma_M(\{1\leq s\leq (1+\delta)^{1/2}\})\int_{S_M}\exp\veps\cdot \vz d\lambda_M(\veps).
\end{align*}
\end{proof}

\noindent From this lemma and Fubini's theorem, the first term in $(\ref{ASS:main:eq1})$ is bounded from below by 
\begin{align}
\label{ASS:lem5.1:eq1}
\mathbb{E}\log\int_{A_\delta}\left<\exp \veps\cdot\vz(\vsi)\right>_{M,N}d\lambda_M(\veps)+\log \nu_M(A_\delta)
\end{align}
Using the standard Gaussian interpolation technique, Lemma $\ref{ASS:lem4}$ below takes care of the second term in $(\ref{ASS:main:eq1})$.

\begin{lemma}
\label{ASS:lem4}
We have
\begin{align}
\label{ASS:lem4:eq1}
\lim_{N\rightarrow\infty}\left|\mathbb{E}\log\frac{ \int_{S_N}\exp H_{N}(\vsi)d\lambda_N(\vsi)}{\int_{S_N}\exp H_{M,N}(\vsi)d\lambda_N(\vsi)}
-\mathbb{E}\log \left<\exp\sqrt{M}y(\vsi)\right>_{M,N}\right|=0.
\end{align}
\end{lemma}

\begin{proof}
Suppose $(y'(\vsi):\vsi\in S_N)$ is a family of centered Gaussian random variables independent of all the other randomness such that in distribution
\begin{align*}
H_N(\vsi)=H_{M,N}(\vsi)+\sqrt{M}y'(\vsi).
\end{align*}
It is easy to check that 
$\mathbb{E}y'(\vsi^1)y'(\vsi^2)=\sum_{p\geq 1}\beta_p^2D_pN^pR(\vsi^1,\vsi^2)^{p}$, where $$
D_{p}=\frac{1}{M}\left(\frac{1}{N^{p-1}}-\frac{1}{(M+N)^{p-1}}\right)
$$ 
satisfies $N^pD_{p}\rightarrow p-1.$ Recall $y(\vsi)$ from $(\ref{ASS:eq0.2}).$ Let us consider the interpolating free energy
\begin{align*}
\psi(t)&=\mathbb{E}\log \left<\exp \sqrt{M}\left(\sqrt{t}y'(\vsi)+\sqrt{1-t}{y}(\vsi)\right)\right>_{M,N}.
\end{align*}
Then $|\psi(1)-\psi(0)|$ gives the quantity in the limit on the left-hand side of $(\ref{ASS:lem4:eq1})$. Thus it suffices to prove that $\psi'(t)\rightarrow 0$ uniformly in $t\in (0,1)$ as $N\rightarrow\infty,$ which can be easily verified by using Gaussian integration by parts, as $N\rightarrow\infty,$
\begin{align*}
\sup_{0<t<1}|\psi'(t)|&\leq M\sup_{\vsi^1,\vsi^2\in S_N}|\mathbb{E}y'(\vsi^1)y'(\vsi^2)-\mathbb{E}y(\vsi^1)y(\vsi^2)|\\
&=M\sup_{\vsi^1,\vsi^2\in S_N}\left|\sum_{p\geq 1}\beta_p^2(D_{p}N^p-(p-1))R(\vsi^1,\vsi^2)^p\right|\rightarrow 0.
\end{align*}
\end{proof}

\noindent Let us continue to find a lower bound for the lower limit of the third term in $(\ref{ASS:main:eq1})$. 
It relies on an elementary lemma that allows us to decouple the surface measure $(S_{M+N},\lambda_{M+N})$ as the product measure $(S_N\times A_\delta,\lambda_N\times F_M)$ asymptotically.
For $K\geq 1,$ we denote by $S_K^1$ the unit sphere in $\mathbb{R}^K$ and by $|S_K^1|$ the area of $S_K^1$.
Set
\begin{align*}
A_{M,N}&=\prod_{j=1}^M\left[-\sqrt{M+N+1-j},\sqrt{M+N+1-j}\right],\\
b_{M,N}&=\prod_{j=1}^M\frac{|S_{M+N-j}^1|}{|S_{M+N+1-j}^1|\sqrt{M+N+1-j}}.
\end{align*}
Define
\begin{align*}
F_{M,N}(\veps)&=b_{M,N}\prod_{j=1}^M\left(1-\frac{\varepsilon_j^2}{M+N+1-j}\right)^{\frac{M+N-j-2}{2}},\quad\veps\in A_{M,N},\\
F_M(\veps)&=\left(\frac{1}{2\pi}\right)^{M/2}\exp\left(-\frac{\|\veps\|^2}{2}\right),\quad\veps\in \mathbb{R}^M.
\end{align*}
Define $a_1=1$ and for $2\leq \ell\leq M+1,$
$$
a_\ell(\veps)=\prod_{j=1}^{\ell-1}\sqrt{1+\frac{1-\varepsilon_j^2}{M+N-j}},\quad \veps\in A_{M,N}.
$$

\begin{lemma}
\label{ASS:lem1}
Suppose that $f$ is a nonnegative function defined on $S_{M+N}.$ 
Then we have 
\begin{align}\label{ASS:lem1:eq1}
\int_{S_{M+N}}f(\vrho)d\lambda_{M+N}(\vrho)
&=\int_{A_{M,N}}
F_{M,N}(\veps)d\boldsymbol{\varepsilon}\int_{S_N}f\left(\vsi a_{M+1},\varepsilon_{1}a_{1},\ldots,\varepsilon_Ma_M\right)d\lambda_{N}(\vsi).
\end{align}
\end{lemma}

\begin{proof}
This is an elementary calculus formula that follows by an induction argument starting from  
\begin{align*}
\int_{S_K}f(\vrho)d\lambda_K(\vrho)&=\frac{|S_{K-1}^1|}{|S_K^1|\sqrt{K}}\int_{-\sqrt{K}}^{\sqrt{K}}\left(1-\frac{\varepsilon^2}{K}\right)^{\frac{K-3}{2}}d\varepsilon \int_{S_{K-1}}f\left(\vsi \sqrt{\frac{K-\varepsilon^2}{K-1}},\varepsilon\right)d\lambda_{K-1}(\vsi)
\end{align*}
for every $K\geq 2.$
\end{proof}

\noindent As an immediate consequence of Lemma \ref{ASS:lem1}, the cavity coordinates converge weakly to the $M$-dimensional standard Gaussian random vector as $N$ tends to infinity. We write $\mathbb{E}\log {Z_{M+N}}/{\int_{A_\delta}J_{M,N}d\nu_M(\veps)}$ as
\begin{align}\label{sec2:eq1}
&\mathbb{E}\log \frac{Z_{M+N}}{\int_{A_\delta}F_{M,N}J_{M,N}d\veps}+\mathbb{E}\log \frac{\int_{A_\delta}F_{M,N}J_{M,N}d\veps}{\int_{A_\delta}F_MJ_{M,N}d\veps}.
\end{align}
Then the first term of $(\ref{sec2:eq1})$ can be controlled by the following lemma.
  
\begin{lemma}\label{ASS:lem2} We have 
\begin{align}
\label{ASS:lem2:eq1}
&\liminf_{N\rightarrow\infty}\mathbb{E}\log \frac{Z_{M+N}}{\int_{A_\delta}F_{M,N}J_{M,N}d\veps}
\geq  -M\delta\xi'(1).
\end{align}
\end{lemma}

\begin{proof}
Suppose that $(H_{M,N}'(\vsi):\vsi\in S_N)$ is an independent copy of $(H_{M,N}(\vsi):\vsi\in S_N)$ and is independent of all the other randomness.
We consider the following interpolating Hamiltonian
\begin{align*}
\varphi(t)=&\mathbb{E}\log\int_{A_\delta}F_{M,N}(\veps)d\veps\int_{S_N}\exp(H_{1,t}(\vrho)+H_{2,t}(\vrho)+H_{3,t}(\vrho))d
\lambda_N(\vsi),
\end{align*}
where for $\vrho=(\vsi,\veps)$ with $\vsi\in S_N$ and $\veps\in A_\delta,$
\begin{align*}
H_{1,t}(\vrho)&=\sqrt{t}H_{M,N}(\vsi a_{M+1})+\sqrt{1-t}H_{M,N}'(\vsi)\\
H_{2,t}(\vrho)&=\sum_{i\leq M}\varepsilon_i\left(\sqrt{t}a_i Z_i(\vsi a_{M+1})+\sqrt{1-t}z_i(\vsi)\right)\\
H_{3,t}(\vrho)&=\sqrt{t}\gamma(\vsi a_{M+1},\varepsilon_1a_1,\ldots,\varepsilon_Ma_M).
\end{align*}
From $(\ref{ASS:lem1:eq1})$, obviously $\varphi(1)\leq\mathbb{E}\log Z_{M+N}$ when $N$ is large enough. On the other hand, using Gaussian integration by parts and replicas $\vrho^1=(\vsi^1,\veps^1)$ and $\vrho^2=(\vsi^2,\veps^2),$ it follows that
\begin{align*}
\varphi'(t)=\frac{1}{2}\sum_{j\leq 3}\left(\mathbb{E}\left<U_{j}(\vrho,\vrho)\right>_t-\mathbb{E}\left<U_j(\vrho^1,\vrho^2)\right>_t\right),
\end{align*}
where $\left<\cdot\right>_t$ is the Gibbs average corresponding to random weights $F_{M,N}\exp(H_{1,t}+H_{2,t}+H_{3,t})$ with reference measure
$\lambda_N\times d\veps$ on $S_N\times A_\delta$ and
\begin{align*}
U_1(\vrho^1,\vrho^2)&=\mathbb{E}H_{M,N}(\vsi^1a_{M+1}(\veps^1))H_{M,N}(\vsi^2a_{M+1}(\veps^2))-
\mathbb{E}H_{M,N}'(\vsi^1)H_{M,N}'(\vsi^2)\\
U_2(\vrho^1,\vrho^2)&=\sum_{i\leq M}\varepsilon_i^1\varepsilon_i^2\left(a_i(\veps^1)a_{i}(\veps^2)\mathbb{E}Z_i(\vsi^1a_{M+1}(\veps^1))Z_i(\vsi^2a_{M+1}(\veps^2))
-\mathbb{E}z_i(\vsi^1)z_i(\vsi^2)\right)\\
U_3(\vrho^1,\vrho^2)&=\mathbb{E}\gamma(\vsi^1a_{M+1}(\veps^1),\eps_1^1a_1(\veps^1),\ldots,\eps_M^1a_M(\veps^1))\\
&\qquad\cdot\gamma(\vsi^2a_{M+1}(\veps^2),\eps_1^2a_1(\veps^2),\ldots,\eps_M^2a_M(\veps^2)).
\end{align*}
From $(\ref{ASS:eq0.1})$, $(\ref{ASS:eq2.2}),$ and $(\ref{ASS:eq2.3})$, it is easy to see that for $j=2,3,$
$$\lim_{N\rightarrow\infty}\sup_{\vrho^1,\vrho^2\in S_N\times A_\delta}|U_j(\vrho^1,\vrho^2)|\rightarrow 0$$
and from $(\ref{ASS:eq2.1})$ and using mean value theorem,
\begin{align*}
&\lim_{N\rightarrow\infty}\sup_{\vrho^1,\vrho^2\in S_N\times A_\delta}|U_1(\vrho^1,\vrho^2)|\\
&\leq \lim_{N\rightarrow\infty}\xi'(1)(M+N)\sup_{\veps^1,\veps^2\in A_\delta}\left|\frac{N}{M+N}a_{M+1}(\veps^1)a_{M+1}(\veps^2)-1\right|\\
&=\xi'(1)\sup_{\veps^1,\veps^2\in A_\delta}\left|M-\frac{\|\veps^1\|^2+\|\veps^2\|^2}{2}\right|\leq M\delta\xi'(1).
\end{align*}
Thus we conclude that $\limsup_{N\rightarrow\infty}\sup_{0<t<1}|\varphi'(t)|\leq M\delta\xi'(1)$ and so $(\ref{ASS:lem2:eq1})$ follows by
\begin{align*}
\liminf_{N\rightarrow\infty}\mathbb{E}\log \frac{Z_{M+N}}{\int_{A_\delta}F_{M,N}J_{M,N}d\veps}&\geq \liminf_{N\rightarrow\infty}(\varphi(1)-\varphi(0))\\
&\geq -\limsup_{N\rightarrow\infty}\sup_{0<t<1}|\varphi'(t)|\geq -M\delta\xi'(1).
\end{align*}
\end{proof}

\noindent Finally, let us deal with the second term of $(\ref{sec2:eq1})$ below.

\begin{lemma}\label{ASS:lem3} We have
\begin{align}\label{ASS:lem3:eq1}
\liminf_{N\rightarrow\infty}\mathbb{E}\log\frac{\int_{A_\delta}F_{M,N}J_{M,N}d\veps}{\int_{A_\delta}F_{M}J_{M,N}d\veps}\geq 0.
\end{align}

\end{lemma}

\begin{proof}
Using the inequality $\log(1-x)\geq -2x$ for $0<x<1/2,$ when $N$ is sufficiently large, we have for every $\veps\in A_\delta,$
\begin{align*}
&\log \frac{F_{M,N}(\veps)}{F_M(\veps)}\\
&=\log b_{M,N}-\frac{M}{2}\log\frac{1}{2\pi}+\sum_{j=1}^M\frac{M+N-j-2}{2}\log\left(1-\frac{\varepsilon_j^2}{M+N+1-j}\right)+\frac{\|\veps\|^2}{2}\\
&\geq \log b_{M,N}-\frac{M}{2}\log\frac{1}{2\pi}+\frac{3\|\veps\|^2}{M+N}.
\end{align*}
Note that $\lim_{N\rightarrow\infty}b_{M,N}=1/(2\pi)^{M/2}$ as can be seen from $(\ref{ASS:lem1:eq1})$ with $f=1.$ By Jensen's inequality, we obtain
\begin{align*}
&\liminf_{N\rightarrow\infty}\mathbb{E}\log\frac{\int_{A_\delta}F_MJ_{M,N}\exp \left(\log F_{M,N}/F_{M}\right)d\veps}{\int_{A_\delta}F_MJ_{M,N}d\veps}\\
&\geq \liminf_{N\rightarrow\infty}\mathbb{E}\frac{\int_{A_\delta}F_MJ_{M,N}\log F_{M,N}/F_{M}d\veps}{\int_{A_\delta}F_MJ_{M,N}d\veps}\\
&\geq \liminf_{N\rightarrow\infty}\left(\log b_{M,N}-\frac{M}{2}\log \frac{1}{2\pi}+\frac{3M}{M+N}\right)=0
\end{align*}
and this gives $(\ref{ASS:lem3:eq1}).$
\end{proof}

\smallskip

\noindent{\bf Proof of Theorem \ref{ASS:prop1}.}
From $(\ref{ASS:main:eq0})$, $(\ref{ASS:main:eq1})$, and $(\ref{sec2:eq1}),$ using Lemmas  $\ref{ASS:lem4},$ $\ref{ASS:lem2}$, and $\ref{ASS:lem3},$ the lower limit of 
$N^{-1}{\mathbb{E}\log Z_{N}}$ is bounded from below by
\begin{align*}
\frac{1}{M}\liminf_{N\rightarrow\infty}\left(\mathbb{E}\log\int_{A_\delta}\left<\exp \veps\cdot\vz(\vsi)\right>_{M,N}d\nu_M(\veps)-\mathbb{E}\log \left<\exp\sqrt{M}y(\vsi)\right>_{M,N}\right)-\delta\xi'(1)
\end{align*}
and then  $(\ref{ASS:lem5.1:eq1})$ yields $(\ref{ASS:prop1:eq1}).$\qed


\section{Proof of the Parisi formula for spherical models.}\label{proof}

\noindent This section is devoted to proving Theorem $\ref{thm1}.$ Our argument is started with Guerra's replica symmetry breaking interpolation scheme in the spherical models (see Section 3 in \cite{Tal061}). In the case when $p$-spin interactions for odd $p\geq 3$ are not present in $(\ref{eq2})$, this scheme implies that the upper limit of the free energy is bounded from above by $\inf_{\vx}\mathcal{P}(\vx).$ The fact that this bound still holds in the general mixed $p$-spin case relies on Talagrand's positivity principle on the overlap (see \cite{Tal03} and Section 12.3 in \cite{Tal11}) that is deduced from the validity of the G.G. identities \cite{GG98} derived by a perturbation trick on the Hamiltonian. This part of the argument is well-known and has been used in several places. For example, the readers are referred to Theorem 14.4.4 \cite{Tal11} for a great detailed discussion.

Our main concern in the proof will be concentrated on obtaining the lower bound, that is, the lower limit of the free energy is bounded from below by
$\inf_{\vx}\mathcal{P}(\vx)$, that will be established in several steps and is based on the A.S.S. scheme. Again due to the non-product structure of the uniform measure on the sphere, additional technical issues occur. For instance, in the Ising spin glass model, the argument for the lower bound only uses the A.S.S. scheme with a single cavity coordinate, that is $M=1$ (see Theorem 3.5 \cite{AC10}, Lemma 11 \cite{Pan10:2}, or Proposition 2 \cite{Pan11:1}), but this is not the case in our model. Indeed, recalling the formulation of the Parisi formula in the spherical model, one may figure out immediately that it is necessary to deal with the A.S.S. scheme $(\ref{ASS:prop1:eq1})$ with all $M\geq 1$ and compute their lower limit in $M.$ Under this circumstance, we find the fact that the Parisi functionals $(\mathcal{P}_M)_{M\geq 1}$ have the same Lipschitz constant with respect to the $L_1$-norm as stated in Lemma $\ref{AGM:prop21}$ below is of great use. 
\smallskip

\noindent{\bf The A.S.S. scheme for the perturbed Hamiltonian.} For each $u=(u_p)_{p\geq 1}$ with $1\leq u_p\leq 2,$ let us consider a perturbation Hamiltonian 
\begin{align*}
H_{N}^{pert}(\vsi)=\frac{1}{N^{1/8}}\sum_{p\geq 1}\frac{u_p}{2^p}H_{N,p}'(\vsi),
\end{align*}
where $H_{N,p}'(\vsi)$ are independent copies of the $p$-spin Hamiltonian in $(\ref{eq1})$. Since the perturbation term is of a small order, replacing $H_N$ with $H_N+H_N^{pert}$ in $(\ref{thm1:eq1})$ obviously does not affect the limit and in addition, the following A.S.S. scheme also holds. 
Consider another perturbation Hamiltonian
\begin{align*}
H_{M,N}^{pert}(\vsi)=\frac{1}{(M+N)^{1/8}}\sum_{p\geq 1}\frac{u_p}{2^p}\left(\frac{N}{M+N}\right)^{(p-1)/2}H_{N,p}'(\vsi).
\end{align*}
Let $G_{M,N}^-$ be the Gibbs measure corresponding to the Hamiltonian $H_{M,N}+H_{M,N}^{pert}$ with reference measure $\lambda_N$ and $\left<\cdot\right>_{M,N}^-$ be the average with respect to the product Gibbs measure $G_{M,N}^{-\otimes\infty}.$

\begin{proposition}[The A.S.S. scheme]\label{ASS:cor} Recall $A_\delta$ and $\nu_M$ from Theorem $\ref{ASS:prop1}.$
For every $\delta>0$ and $M\geq 1,$ we have
\begin{align}
\begin{split}\label{ASS:cor:eq1}
&\liminf_{N\rightarrow\infty}\frac{1}{N}\mathbb{E}\log Z_N\\
&\geq\frac{1}{M}\liminf_{N\rightarrow\infty}\sup_{u}\left(\mathbb{E}\log \int_{S_M}\left<\exp \veps\cdot \vz (\vsi)\right>_{M,N}^-d\lambda_M(\veps)-
\mathbb{E}\log \left<\exp \sqrt{M}y(\vsi)\right>_{M,N}^-\right)\\
&\qquad\qquad -\delta\xi'(1)+\frac{1}{M}\log \nu_M(A_\delta),
\end{split}
\end{align}
where the supremum is taken over all $u=(u_{p})_{p\geq 1}$ with $1\leq u_p\leq 2.$
\end{proposition}

\smallskip

\noindent Starting with the Hamiltonian $H_{M+N}+H_{M+N}^{pert}$ and using the fact that $H_{M+N}^{pert}$ is of a small order, one may easily find that the same arguments in the proof of Theorem $\ref{ASS:prop1}$ also work in Proposition $\ref{ASS:cor}$, so we will leave this part of the proof to the readers. In what follows, for each $M\geq 1$, we will construct a sequence $(u^{M,N})_{N\geq 1}$ for $u^{M,N}=(u_p^{M,N})_{p\geq 1}$ with $1\leq u_p^{M,N}\leq 2$ along which the G.G. identities for the Gibbs measure $G_{M,N}^-$ and the A.S.S. scheme for the Hamiltonian $H_{M,N}+H_{M,N}^{pert}$ hold simultaneously. Let $(\vsi^\ell)_{\ell\geq 1}$ be an i.i.d. sample from $G_{M,N}^-$ and let $R^{M,N}=(R_{\ell,\ell'}^{M,N})_{\ell,\ell'\geq 1}$ be the normalized Gram matrix, or matrix of overlaps, of this sample. Let $\mathcal{C}$ be the collection of all $(p,n,f)$ with $p\geq 1,$ $n\geq 2,$ and $f$ being a monomial of $(R_{\ell,\ell'}^{M,N})_{\ell,\ell'\leq n}.$ We define
\begin{align*}
&\Phi_{M,N}(p,n,f)\\
&=\left|\mathbb{E}\left<f(R_{1,n+1}^{M,N})^p\right>_{M,N}^--\frac{1}{n}\mathbb{E}\left<f\right>_{M,N}^-
\mathbb{E}\left<(R_{1,2}^{M,N})^p\right>_{M,N}^--\frac{1}{n}\sum_{\ell=2}^n
\mathbb{E}\left<f(R_{1,\ell}^{M,N})^p\right>_{M,N}^-\right|
\end{align*}
for $(p,n,f)\in \mathcal{C}.$ Then the G.G. identities can be stated as follows (see Theorem 12.3.1 \cite{Tal11}).

\begin{proposition}\label{prop:GG}
For $M\geq 1$, we have 
\begin{align}\label{prop:GG:eq1}
\lim_{N\rightarrow\infty}\mathbb{E}_u\Phi_{M,N}(p,n,f)=0
\end{align}
for all $(p,n,f)\in\mathcal{C},$ where we view $u=(u_p)_{p\geq 1}$ as a sequence of i.i.d. uniform random variables on $\left[1,2\right]$ and $\mathbb{E}_u$ is the expectation with respect to the randomness of $u.$
\end{proposition}
\smallskip

\noindent Since $\mathcal{C}$ is a countable set, we can enumerate it as $\left\{(p_j,n_j,f_j)\right\}_{j\geq 1}.$
Let us define 
\begin{align*}
\Phi_{M,N}(u)&=\sum_{j\geq 1}\frac{1}{2^j}\Phi_{M,N}(p_j,n_j,f_j),\\
\Lambda_{M,N}(u)&=\frac{1}{M}\left(\mathbb{E}\log\int_{S_M}\left<\exp \veps\cdot\vz(\vsi)\right>_{M,N}^-d\lambda_M(\veps)
-\mathbb{E}\log \left<\exp \sqrt{M}y(\vsi)\right>_{M,N}^-\right).
\end{align*} 
By Chebyshev's inequality, $\mathbb{P}_u(\Phi_{M,N}\leq \eta)\geq 1-\mathbb{E}_u\Phi_{M,N}/\eta$ for every $\eta>0.$ If we pick $\eta=2\mathbb{E}_u\Phi_{M,N},$ it implies that for any $M$ and $N$ we can always find $u^{M,N}=(u_p^{M,N})_{p\geq 1}$ such that 
$\Phi_{M,N}(u^{M,N})\leq 2\mathbb{E}_u\Phi_{M,N}$. Note that using $\Phi_{M,N}(n,p,f)\leq 2$ and the G.G. identities $(\ref{prop:GG:eq1})$ implies $\lim_{N\rightarrow\infty}\mathbb{E}_u\Phi_{M,N}=0.$ Consequently, we have
\begin{align}
\begin{split}\label{AGM:eq2}
\lim_{N\rightarrow\infty}\Phi_{M,N}(u^{M,N})&=0
\end{split}
\end{align}
and from $(\ref{ASS:cor:eq1})$, $\liminf_{N\rightarrow\infty}N^{-1}\mathbb{E}\log Z_N$ is bounded from below by
\begin{align}
\begin{split}\label{AGM:eq2.1}
\liminf_{N\rightarrow\infty}\Lambda_{M,N}(u^{M,N})-\delta\xi'(1)+\frac{1}{M}\log \nu_M(A_\delta).
\end{split}
\end{align} 
In other words, for each $M\geq 1,$ both the G.G. identities and the A.S.S. scheme hold along $(u^{M,N})_{N\geq 1}$.
Let us redefine the Hamiltonian $H_{M,N}^{pert}$ and the Gibbs measure $G_{M,N}^-$ by fixing parameters $u=u^{M,N}.$

\smallskip

\noindent {\bf Asymptotic Gibbs' measures.} Next, we will find a sequence of asymptotic Gibbs' measures that satisfy the G.G. identities and represent the lower limit in $(\ref{AGM:eq2.1})$. For each $M\geq 1,$ let us pick a subsequence $(N_n)$ along which the limits in $(\ref{AGM:eq2})$ and $(\ref{AGM:eq2.1})$ are achieved and the distribution $R^{M,N}$ under $\mathbb{E}G_{M,N}^{-\otimes\infty}$ converges in the sense of convergence of finite dimensional distributions to the distribution of some array $R^M$. To lighten the notation, we assume without loss of generality that the sequence $(N_n)$ coincides with the natural numbers. Under $\mathbb{E}G_{M,N}^{-\otimes\infty}$, the array ${R}^{M,N}$ is weakly exchangeable, that is, 
\begin{align*}
\left(R_{\pi(\ell),\pi(\ell')}^{M,N}\right)\stackrel{d}{=}\left(R_{\ell,\ell'}^{M,N}\right)
\end{align*}
for any permutation $\pi$ of finitely many indices. This property will still be preserved in the limit such that $R^M$ is a weakly exchangable symmetric non-negative definite array. Such an array is called a Gram-de Finetti array and is known to have the Dovbysh-Sudakov representation \cite{DS82}.

\begin{proposition}[The Dovbysh-Sudakov representation]\label{AGM:Prop1}
If $(R_{\ell,\ell'})_{\ell,\ell'\geq 1}$ is a Gram-de Finetti array such that $R_{\ell,\ell}=1$, then there exists a random measure $G$ on the unit ball of a separable Hilbert space such that
\begin{align*}
\left(R_{\ell,\ell'}\right)_{\ell,\ell'\geq 1}\stackrel{d}{=}\left(\vtau^{\ell}\cdot\vtau^{\ell'}+\delta_{\ell,\ell'}(1-\|\vtau^{\ell}\|^2)\right)_{\ell,\ell'\geq 1},
\end{align*}
where $(\vtau^\ell)$ is an i.i.d. sample from $G.$
\end{proposition}

\smallskip

\noindent Obviously the array $R^M$ satisfies $R_{\ell,\ell}^M=1.$ Using Proposition $\ref{AGM:Prop1}$, let us denote by $G_M$ the random measure generating the array $R^{M}$ and by $(\vtau^\ell)$ an i.i.d. sample from $G_M$. We also denote by $\left<\cdot\right>_M$ the average with respect to $G_M.$ From $(\ref{AGM:eq2})$ and an approximation argument, the measure $G_M$ satisfies the G.G. identities,
\begin{align}\label{AGM:eq4}
\mathbb{E}\left<f\psi(R_{1,n+1}^M)\right>_M=\frac{1}{n}\mathbb{E}\left<f\right>_M\mathbb{E}\left<\psi(R_{1,2}^M)\right>_M+\frac{1}{n}\sum_{\ell=2}^n
\mathbb{E}\left<f\psi(R_{1,\ell}^M)\right>_M
\end{align}
for all $n\geq 2,$ all bounded measurable function $f$ of $(R_{\ell,\ell'}^M)_{\ell,\ell'\leq n}$, and measurable function $\psi.$
These identities will be used to identify the asymptotic Gibbs measures $G_M$ in next subsection.

Let us continue to check that the limit of $\Lambda_{M,N}(u^{M,N})$ in $N$ can be expressed in terms of $G_{M},$ which relies on the following observation. Suppose for the moment that $R=(R_{\ell,\ell'})_{\ell,\ell'\geq 1}$ is an arbitrary Gram-de Finetti array such that $R_{\ell,\ell}=1.$ Let $\mathcal{L}$ be its distribution  and $G$ be any random measure generating $R$ as Proposition $\ref{AGM:Prop1}.$ Suppose that $\vz(\vtau)=(z_i(\vtau))_{i\leq M}$ and $y(\vtau)$ are two Gaussian processes on the unit ball of some Hilbert space generating $R$ with covariances
\begin{align*}
\mathbb{E}z_i(\vtau^1)z_{i'}(\vtau^2)=\delta_{i,i'}\xi'(\vtau^1\cdot\vtau^2),\,\,\mathbb{E}y(\vtau^1)y(\vtau^2)=\theta(\vtau^1\cdot\vtau^2),
\end{align*}
let $\veta$ be a $M$-dimensional standard Gaussian random vector, and let $\eta$ be a standard Gaussian random variable. Denote by $\left<\cdot\right>$ the Gibbs average corresponding to $G$ and define
\begin{align*}
\Psi_M(\mathcal{L})=&\frac{1}{M}\mathbb{E}\log \int_{S_M}\mathbb{E}_{\veta}\left<\exp\veps\cdot(\vz (\vtau)+\veta(\xi'(1)-\xi'(\|\vtau\|^2))^{1/2})\right>d\lambda_M(\veps)\\
&-\frac{1}{M}\mathbb{E}\log \mathbb{E}_\eta\left<\exp\sqrt{M}({y}(\vtau)+\eta(\xi'(1)-\xi'(\|\vtau\|^2))^{1/2})\right>.
\end{align*}

\begin{lemma}\label{AGM:lem2}
For each $M\geq 1,$ $\mathcal{L}\rightarrow \Psi_M(\mathcal{L})$ is a well-defined continuous function with respect to the weak convergence of the distribution $\mathcal{L}.$
\end{lemma}

\smallskip

\noindent This is a generalization of Lemma 3 in \cite{Pan11:1} and can be verified exactly in the same way.  
Let $\mathcal{L}^{M,N}$ and $\mathcal{L}^M$ be the distributions generating $R^{M,N}$ and $R^M$, respectively. 
Then Lemma $\ref{AGM:lem2}$ leads to
\begin{align}\label{AGM:eq8}
\lim_{N\rightarrow\infty}\Lambda_{M,N}(u^{M,N})&=\lim_{N\rightarrow\infty}\Psi_M(\mathcal{L}^{M,N})=\Psi_M(\mathcal{L}^M).
\end{align}


\noindent{\bf Identitying the asymptotic Gibbs' measures using ultrametricity.} Since the asymptotic Gibbs measure $G_M$ satisfies the G.G. identities $(\ref{AGM:eq4})$, the main result in $\cite{Pan11:4}$ implies that the support of $G_M$ is ultrametric with probability one, that is, 
\begin{align*}
\mathbb{E}\left<I(R_{1,2}^M\geq \min(R_{1,3}^M,R_{2,3}^M))\right>_M=1.
\end{align*}
This implies that for any $q$, the inequality $q\leq \vtau^{\ell}\cdot\vtau^{\ell'}$ defines an equivalent relation and therefore, the 
array $(I(q\leq R_{\ell,\ell'}^M))_{\ell,\ell'\geq 1}$ is nonnegative definite since it is block-diagonal with blocks consisting of all elements
equal to one. Consider a sequence of functions $(H_k)_{k\geq 1}$ on $\left[0,1\right]$ by defining $H_{k}(q)=j/k$ if $j/(k+1)\leq q<(j+1)/(k+1)$ for $0\leq j\leq k$ and $H_k(1)=1.$ Then we see that $R_{k}^M=(H_k(R_{\ell,\ell'}^M))_{\ell,\ell\geq 1}$ is symmetric, weakly exchangable, and satisfies the G.G. identities. In addition, it is non-negative definite since it can be written as a convex combination of nonnegative definite arrays,
\begin{align*}
H_k(R_{\ell,\ell'}^M)=\sum_{j=1}^{k}\frac{1}{k}I\left(\frac{j}{k+1}\leq R_{\ell,\ell'}^M\right).
\end{align*}
Thus by the Dovbysh-Sudakov representation, $R_{k}^M$ can be generated by a sample from a random measure $G_M^k$ on the unit ball of some Hilbert space.
Since $\sup_{0\leq q\leq 1}|H_k(q)-q|\leq 1/k,$ $(R_k^M)_{k\geq 1}$ converges weakly to $R^M.$ From Lemma $\ref{AGM:lem2},$ we then obtain
\begin{align}\label{AGMU:add1}
\lim_{k\rightarrow\infty}\Psi_M(\mathcal{L}_k^M)=\Psi_M(\mathcal{L}^M),
\end{align}  
where $\mathcal{L}_k^M$ is the distribution for $R_k^M$. On the other hand, from the Baffioni-Rosati theorem \cite{BR00}, an ultrametric measure, such as $G_M^k$, that satisfies the G.G. identities and under which the overlaps $R_k^M$ take finitely many values $(p/k)_{0\leq p\leq k}$ can be identified as the Poisson-Dirichlet cascade with parameters 
\begin{align*}
\mathbf{m}^{M,k}&=(m_p^{M,k})_{0\leq p\leq k}\,\,\mbox{with $m_p^{M,k}=\mathbb{E}\left<I((R_k^M)_{1,2}\leq p/k)\right>_M,$}\\
\mathbf{q}^{M,k}&=(q_p^{M,k})_{0\leq p\leq k+1}\,\,\mbox{with $q_p^{M,k}=p/(k+1)$}.
\end{align*}
For details of this result, one may consult the proof of Theorem 15.3.6 in \cite{Tal11}. Let $\vx_k^M$ be the distribution function corresponding to this triplet $(k,\mathbf{m}^{M,k},\mathbf{q}^{M,k}).$ Using the same argument as Theorem 14.2.1 in \cite{Tal11}, we obtain
\begin{align}\label{AGM:eq6}
\Psi_M(\mathcal{L}_k^M)=\mathcal{P}_M(\vx_k^M).
\end{align}

\begin{lemma} \label{AGM:prop21}
For any two triplets $(k,\mathbf{m},\mathbf{q})$ and $(k',\mathbf{m}',\mathbf{q}')$ associated with the distribution functions $\vx$ and $\vx'$ respectively, we have
\begin{align}
\label{AGM:prop2}
\left|\mathcal{P}_M(\vx)-\mathcal{P}_M(\vx')\right|\leq \frac{\xi'(1)}{2}\int_0^1\left|\vx(q)-\vx'(q)\right|dq.
\end{align} 
\end{lemma}

\smallskip

\noindent 
It is well-known that the same Lipschitz property as $(\ref{AGM:prop2})$ for the Parisi functional in the Ising spin glass model was established by F. Guerra \cite{Guerra03}. In the spherical models, despite of the non-product structure on the configuration space, one may find the same approach as Theorem 14.11.2 in \cite{Tal11} can also be applied to obtain $(\ref{AGM:prop2}).$ Let $\mathcal{F}$ be the collection of all distribution functions $\vx$ on $\left[0,1\right]$. Define a metric $d$ on $\mathcal{F}$ by letting $d(\vx,\vx')=\int_0^1\left|\vx(q)-\vx'(q)\right|dq$ for every $\vx,\vx'\in \mathcal{F}.$ Using the Lipschitz inequality $(\ref{AGM:prop2})$, we can extend the functional $\mathcal{P}_M$ continuously to $\mathcal{F}.$

\smallskip

\noindent{\bf Proof of Theorem \ref{thm1}.} If we denote by $\vx^M$ the distribution of $R_{1,2}^M,$ it follows that by using $(\ref{AGMU:add1})$ and  $(\ref{AGM:eq6})$,
\begin{align}\label{AGM:eq7} 
\Psi_M(\mathcal{L}^M)=\mathcal{P}_M(\vx^M).
\end{align}
From Helly's selection theorem, $(\vx^M)_{M\geq 1}$ contains a weakly convergent subsequence with limit $\vx^\infty.$ For clarity, we simply use $(\vx^M)$ to denote this subsequence. Using $(\ref{AGM:prop2})$ and $(\ref{AGM:eq7})$, if we set $\vx_k^\infty=H_k(\vx^\infty),$ then
\begin{align*}
\limsup_{M\rightarrow\infty}\left|\Psi_M(\mathcal{L}^M)-\mathcal{P}_M(\vx_k^{\infty})\right|
&=\limsup_{M\rightarrow\infty}\left|\mathcal{P}_M(\vx^M)-\mathcal{P}_M(\vx_k^{\infty})\right|\\
&\leq \frac{\xi'(1)}{2}\left(\limsup_{M\rightarrow\infty}d(\vx^M,\vx^{\infty})+d(\vx^\infty,\vx_k^{\infty})\right)\\
&=\frac{\xi'(1)}{2}d(\vx^\infty,\vx_k^{\infty})
\end{align*} 
for every $k\geq 1.$ To sum up, from $(\ref{AGM:eq2.1}),$ $(\ref{AGM:eq8})$, and this inequality, we conclude that 
\begin{align*}
\liminf_{N\rightarrow\infty}\frac{1}{N}\mathbb{E}\log Z_N&\geq \limsup_{M\rightarrow\infty}\left( \Psi_M(\mathcal{L}^M)+\frac{1}{M}\log\nu_M(A_\delta)\right)-\delta\xi'(1)\\
&\geq \limsup_{M\rightarrow\infty}\mathcal{P}_M(\vx_k^\infty)-\limsup_{M\rightarrow\infty}\left|\Psi_M(\mathcal{L}^M)-\mathcal{P}_M(\vx_k^\infty)\right|-\delta\xi'(1)\\
&\geq\inf_{\vx}\mathcal{P}(\vx)-\frac{\xi'(1)}{2}d(\vx^\infty,\vx_k^{\infty})-\delta\xi'(1)
\end{align*}
for every $k\geq 1$ and $\delta>0$ and this completes our proof by letting $\delta\rightarrow 0$ and $k\rightarrow\infty.$\qed

\end{document}